\documentclass[a4paper,11pt,english]{amsart}
\usepackage[utf8]{inputenc}
\usepackage{amssymb,color,graphicx}
\usepackage{hyperref}
\usepackage[alphabetic,backrefs]{amsrefs}
\usepackage[english,french]{babel}

\newcommand{\lipone}{\operatorname{Lip}_{1,1}}

\newcommand{\C}{\mathbb{C}}
\DeclareMathOperator{\sgn}{sgn}

\newtheorem{thm}{Theorem}
\newtheorem{coro}[thm]{Corollary}
\newtheorem*{thm*}{Theorem}
\newtheorem{definition}{Definition}
\newtheorem*{remark}{Remark}

\title{Equidistribution and $\beta$ ensembles}
\author[T. Carroll]{Tom Carroll}\address[T. Carroll]{University College 
Cork}\email{t.carroll@ucc.ie}
\author[J. Marzo]{Jordi Marzo}\address[J. Marzo]{Universitat  de 
Barcelona}\email{jmarzo@ub.edu}
\author[X. Massaneda]{Xavier Massaneda}\address[X. Massaneda]{
Universitat  de Barcelona, BGSMath}\email{xavier.massaneda@ub.edu}
\author[J. Ortega-Cerdà]{Joaquim Ortega-Cerdà}\address[J. 
Ortega-Cerdà]{Universitat  de Barcelona, BGSMath}\email{jortega@ub.edu}

\begin{document}
\maketitle
\selectlanguage{english}
\begin{abstract}
 We find the precise rate at which the empirical measure associated to a 
$\beta$-ensemble converges to its limiting measure. In our setting the 
$\beta$-ensemble is a random point process on a compact complex  manifolds 
distributed according to the $\beta$ power of a determinant of sections in a 
positive line bundle. A particular case is the spherical ensemble of generalized 
random eigenvalues of pairs of matrices with independent identically distributed 
Gaussian entries. \end{abstract}


\section{Background and setting}
Let $(X,\omega)$ be a $n$-dimensional compact complex manifold endowed with a 
smooth Hermitian metric $\omega$. Let $(L,\phi)$ be a holomorphic line bundle 
with a postive Hermitian metric $\phi$. This has to be understood as a 
collection of smooth functions $\phi_i$ defined in trivializing neighborhoods 
$U_i$ of the line bundle. If $e_i(x)$ is a frame in $U_i$, then 
$\|e_i(x)\|^2=e^{-\phi_i(x)}$. Thus $\phi_i$ must satisfy the compatibilty 
condition $\phi_i-\phi_j=\log|g_{ij}|$, where $g_{ij}$ are the transition 
functions.

As usual we denote by $H^0(L,X)$ the global holomorphic sections.
If $L$ is a line bundle over $X$ and $M$ is a line bundle over $Y$, we denote 
by $L\boxtimes M$ the line bundle over the product manifold $X\times Y$ defined 
as $L \boxtimes M= \pi_X^*(L)\otimes \pi_Y^*(M)$, where $\pi_X: X\times Y\to X$ 
is the projection onto the first factor and $\pi_Y: X\times Y\to Y$ is the 
projection onto the second. The line bundle $L \boxtimes M$ carries a metric 
induced by that of $L$ and $M$.

Given a basis $s_1,\ldots, s_N$ of $H^0(L,X)$ we define $\det(s_i(x_j))$ as a 
section of $L^{\boxtimes N}$ over $M^N$ by the identities 
$\det(s_i(x_j))=\sum_{\sigma\in S_n} \sgn(\sigma) \bigotimes_{i=1}^N 
s_i(x_{\sigma_i})$.

We fix a probability measure on $X$, given by the normalized volume form 
$\omega^n$, that we denote by $\sigma$. 

\begin{definition}\label{def1}
Let $\beta>0$. A $\beta$-ensemble is an 
$N$ point random process on $X$ which has joint distribution given by
\begin{equation}\label{distribution}
\frac 1{Z_N} \|\det s_i(x_j)\|^\beta \,d\sigma(x_1)\otimes \cdots \otimes 
d\sigma(x_N),
\end{equation} 
where $Z_N=Z_N(\beta)$ is chosen so that this is a probability distribution in 
$X^N$.
\end{definition}

Observe that the random point process is independent of the choice of basis 
$s_j$. 		

A particularly interesting case is when $\beta=2$, since then the process is 
determinantal. Let $K$ denote the Bergman kernel of the Hilbert space $H^0(L,X)$ 
endowed with the norm $\| s\| = \int_X \|s(x)\|_\phi^2 \,d\sigma(x)$. Here 
$\|\cdot\|_\phi^2$ denotes the norm induced by the metric $\phi$ (see 
Section~\ref{sec2} for more details). Then
\[
\|\det(s_i(x_j))\|^2=\|\det(K(x_i,x_j))\|.
\]

Another interesting situation occurs when $\beta\to\infty$. In this case the 
probability charges the maxima of the function $\|\det(s_i(x_j))\|$. A set of 
points $\{x_j\}_j$ with cardinality $\dim H^0(L,X)$ and maximizing this 
determinant is known as a Fekete sequence. The distribution of these sequences 
has been studied in \cite{LOC10}, \cite{JM}, and \cite{BBW11} and we will draw 
some ideas from there to study general $\beta$-ensembles.

We consider now the situation where we replace $L$ by a power $L^k$, 
$k\in\mathbb N$, and let $k$ tend to infinity. We denote by $N_k$ the dimension 
of $H^0(L^k,X)$. It is well-known, by the Riemann-Roch theorem and the Kodaira 
vanishing theorem, that
\[
 \dim H^0(L^k,X)= \frac{c_1(L)^n}{n!}  k^n + O(k^{n-1}) \sim k^n ,
\]
where $c_1(L)$ denotes the first Chern class of $L$.

For each $k$ we consider a collection of $N_k$ points chosen randomly according 
to the law \eqref{distribution}. For each $k$ the collection is picked 
independently of the previous ones. 

Given points $x_1^{(k)},\ldots,x_{N_k}^{(k)}$ chosen according to 
\eqref{distribution}, consider its associated empirical measure $\mu_k=\frac 
1{N_k} \sum \delta_{x_i^{(k)}}$. For convenience we will drop the superindex 
$(k)$ hereafter. We are interested in understanding the 
limiting distribution of the measures $\mu_k$. 

The following result is well known; see \cite{BBW11}.

\begin{thm*}[Berman, Boucksom, Witt]
\label{fekete_distribution}
Let $\mu_k$ be the empirical measure associated to a Fekete sequence for the 
bundle $ H^0(L^k,X)$. Then, as $k\to\infty$,
\begin{equation*}
\mu_k \longrightarrow \nu:=\frac{(i \partial \bar\partial \phi)^n}{\int_X 
(i \partial \bar\partial \phi)^n} 
\end{equation*}
in the weak-$*$ topology.
\end{thm*}

The measure $\nu$ is called the equilibrium measure.

There is a counterpart of this result for empirical measures of general 
$\beta$-ensembles (see \cite{Berman14}, which gives an estimate for the large 
deviations of the empirical measure from the equilibrium measure). 
Our aim is to obtain a different quantitative version of the weak convergence of 
the empirical measure to the equilibrium measure, measured in terms of the 
Kantorovich-Wasserstein distance between mesaures.

This sort of quantification has also been studied, with different tools, in the 
context of random matrix models,  (see for instance  \cite{MeMe1}, \cite{MeMe2} 
and \cite{MeMe14}, where similar determinantal point processes arise). 

In fact some of the $\beta$-ensembles we are considering admit random matrix 
models, at least when $\dim_{\C} (M)=1$. For instance,  Krishnapur studied in 
\cite{Krishnapur} the following point process: let $A,B$ be $k\times k$ random 
matrices with i.i.d.\ complex Gaussian entries. He proved that the generalized 
eigenvalues associated with the pair $(A,B)$, i.e.\ the eigenvalues of 
$A^{-1}B$, have joint probability density:
\begin{equation}\label{densitat}
  \frac 1{Z_k} \prod_{l= 1}^k 
\frac{1}{(1+|x_l|^2)^{k+1}}\prod_{i<j}{|x_i-x_j|^2},
\end{equation}
with respect to the Lebesgue measure in the plane.

It was also observed in \cite{Krishnapur} that, using the stereographic 
projection 
\begin{align*}
 \pi:\, & \mathbb S^2\longrightarrow \mathbb C \\
 & P_j \ \mapsto x_j ,
\end{align*}
the joint density \eqref{densitat} (with respect to 
the product area measure in the product of spheres) is 
\[
 \frac 1{Z_k} \prod_{i<j} \|P_i-P_j\|^2_{\mathbb R^3}.
\]
Since this is invariant under rotations of the sphere, the point 
process is called the spherical ensemble.  

A point process with this law had been considered earlier -- without a random 
matrix model -- by Caillol \cite{Caillol} as the model of one-component plasma. 

 One typical instance of the process is as in the picture.
\begin{figure}[h]
\includegraphics[width=12cm]{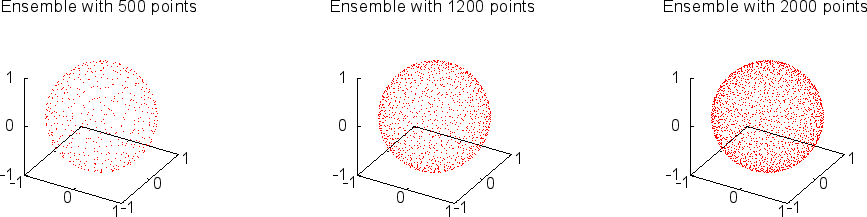}
\end{figure}

The spherical ensemble has received much attention. We mention a couple of 
properties related to our results. In \cite{Bordenave}, Bordenave proves the 
universality of the spectral distribution of the $k\times k$-matrix $A^{-1}B$ 
with respect to other i.i.d.\ random distribution of entries. As an outcome, he 
proves that the weak-* limit of the spectral measures $\mu_k=\frac 1k 
\sum_i\delta_{x_i}$, where $x_i$ are the generalized eigenvalues is 
the normalized area measure in the sphere. This convergence is rather uniform: 
in \cite{AliSade14} Alishahi and Sadegh Zamani estimate the discrepancy of the 
empirical measure with respect to its limit and give precise estimates of the 
Newtonian and the logarithmic energies.

\subsection{The Kantorovich-Wasserstein distance}
To measure the uniformity and speed of convergence of the empirical measures 
$\mu_k$ to the limiting measure $\nu$  we use the Kantorovich-Wasserstein 
distance $W$. Given probability measures $\mu$ and $\nu$, it is defined as 
\[
W(\mu,\nu)=\inf_{\rho }\iint_{X\times X} d(x,y) \, d\rho(x,y),
\]
where $d(x,y)$ is the distance associated to the metric $\omega$ and the infimum 
is taken over all admissible transport plans $\rho$, i.e., all probability 
measures in $X\times X$ with  marginal measures $\mu$ and $\nu$ respectively.

In general, the Kantorovich-Wasserstein distance is defined on probability 
measures over a compact metric space $X$, and it metrizes the weak-$*$ 
convergence of measures. 

It was observed in \cite{LOC10} that in the definition of $W$ it is possible to 
enlarge the class of admissible transport plans to complex measures $\rho$ that 
have marginals $\mu$ and $\nu$ respectively. We include the argument for the 
sake of completness.

Let 
\begin{equation}\label{alternative}
 \widetilde{W}(\mu,\nu)=\inf_{\rho\in S}  \iint_{X \times X} d(x,y) \, 
|d\rho(x,y)| ,
\end{equation}
where the infimum is now taken over the set $S$ of all complex  measures $\rho$ 
on 
$X \times X$ with marginals $\rho(\cdot,X) = \mu$ and $\rho(X,\cdot) = \nu$.

In order to see that $ \widetilde{W}(\mu,\nu)= W(\mu,\nu)$, we recall the dual 
formulation of $W$ (see \cite{Villani}*{Formula~(6.3)}):
 \begin{equation}\label{wass-def-dual}
 W(\mu,\nu)=\sup \left\{\Bigl|\int_{X} f \,d(\mu-\nu)\Bigr|: f\in \lipone(X) 
\right\},
 \end{equation}
where $\lipone(X)$ is the collection of all functions $f$  on $X$ satisfying 
$|f(x)-f(y)|\le d(x,y)$.

For any complex measure $\rho$ with marginals $\mu$ and $\nu$ and any 
$f\in \lipone(X)$ we have
\[
 \Bigl|\int_{X} f \,d(\mu-\nu)\Bigr|=\Bigl|\iint_{X\times X} (f(x)-f(y)) 
\,d\rho(x,y)\Bigr|\le \iint_{X\times X} d(x,y)\,|d\rho(x,y)|.
\]
Hence 
\[
W(\mu,\nu)\le\inf_{\rho\in S}  \iint_{X \times X} d(x,y) \, 
|d\rho(x,y)|=\widetilde{W}(\mu,\nu).
\]
The remaining inequality ($\widetilde{W}(\mu,\nu)\le W(\mu,\nu)$) is trivial.

A standard reference for basic facts on 
Kantorovich-Wasserstein distances is the book \cite{Villani}. 

\subsection{Lagrange sections}\label{sec2}
We fix now a basis of sections $s_1,\ldots s_{N_k}$ of $H^0(L^k,X)$.  
Given any collection of points $(x_1,\ldots,x_{N_k})$ we define the Lagrange
sections informally as:
\[
 \ell_j(x)=\frac{\left|\begin{smallmatrix}
 s_1(x_1)&\cdots&s_1(x)&\cdots &s_1(x_{N_k})\\ 
 \vdots&&\vdots&&\vdots\\
s_{N_k}(x_1)&\cdots&s_{N_k}(x)&\cdots &s_{N_k}(x_{N_k})\\
\end{smallmatrix}\right|}
{\left|\begin{smallmatrix}
 s_1(x_1)&\cdots&s_1(x_j)&\cdots &s_1(x_{N_k})\\ 
 \vdots&&\vdots&&\vdots\\
s_{N_k}(x_1)&\cdots&s_{N_k}(x_j)&\cdots &s_{N_k}(x_{N_k})\\
\end{smallmatrix}\right|}
\]
Clearly $\ell_j\in H^0(L^k,X)$ and $\ell_j(x_i)=0$ if $i\ne j$ and 
$\|\ell_j(x_j)\|=1$.

More formally, we proceed as in \cite{LOC10}:
if $e_j(x)$ is a frame in a neighborhood $U_j$ of the point $x_j$, then
the sections $s_i (x)$ are represented on each $U_j$ by scalar functions
$f_{ij}$ such that
$s_i (x)=f_{ij} (x) e_j(x)$. Similarly, the metric $k\phi$ is represented on 
$U_j$
by a smooth real-valued function $k\phi_j$ such that
$\| s_i (x)\|^2= |f_{ij} (x) |^2 e^{-k\phi_j (x)}$.

To construct the Lagrange sections we denote by $A$ the matrix
\[
\big ( e^{-\frac{k}{2}\phi_j (x_j)}f_{ij}(x_j)\big )_{i,j},
\] 
and define
\begin{equation*}
l_j (x) := \frac{1}{\det (A)} \sum_{i=1}^{N_k} (-1)^{i+j} A_{ij} s_i (x),
\end{equation*}
where $A_{ij}$ is the determinant of the submatrix obtained from $A$ by removing
the $i$-th row and $j$-th column. Clearly $l_j\in H^0(L^k,X)$, and it is not 
difficult
to check that $\| l_j (x_i) \| =\delta_{ij},$  $1\leq i,j\leq N_k$.

Notice that if we denote by $\rho_k(x_1,\ldots,x_{N_k})=\|\det s_i(x_j)\|$ then
\begin{equation}\label{lagra}
 \|\ell_j(x)\|^\beta=
\frac{\rho_k(x_1,\ldots,x,\ldots,x_{N_k})}{\rho_k(x_1,\ldots,x_j,\ldots,x_{N_k})
} .
\end{equation}
and thus $\mathbb E (\|\ell_j\|^\beta) \le 1$. In the case of the 
Fekete points ($\beta=\infty$), $\sup_X \|\ell_j(x)\|=1$ by definition. 

\section{Main result}
\begin{thm}\label{main}
Consider the empirical measure $\mu_k$ associated to the $\beta$-ensemble given 
in Definition~\ref{def1} and let $\nu=\frac{(i \partial \bar\partial 
\phi)^n}{\int_X (i \partial \bar\partial \phi)^n}$ be the equilibrium measure. 
Then
\[ 
 \mathbb{E}W (\mu_k,\nu)=O(1/\sqrt{k}).
\]
\end{thm}

\begin{remark}
The rate of convergence cannot be improved. Let $\sigma$ be any nowhere 
vanishing smooth probability distribution on $X$. Let $E_k$ be any discrete set on 
$X$ with cardinality $\# E_k\simeq k^n \simeq N_k$, and let  $\mu_k=\frac 1{\# 
E_k}\sum_{x\in E_k} \delta_x$. Then the distance $W(\mu_k,\sigma)\geq   
1/\sqrt{k}$.

To obtain a lower bound for $ W(\mu_k,\sigma)$ we use the dual formulation  of 
the Kantorovich-Wasserstein distance \eqref{wass-def-dual} and the function 
$f(x)=d(x,E_k)$, which is in $\lipone(X)$. Since $d(x,E_k)=0$ on the support of 
$\mu_k$ we obtain
\[
 W(\mu_k,\sigma)\ge \int_{X}  d(x,E_k)\,d\sigma.
\]
Vitali's covering lemma ensures that for each $k$  and for some $\varepsilon$ 
small enough, independent of $k$, there are at least $2 \#E_k$ pairwise disjoint 
balls of radius $\varepsilon/\sqrt{k}$. Since the number of balls is twice the 
number of points in $E_k$, at least half the balls contain no point 
of $E_k$. We consider one such ball, $B(x_i,\varepsilon/\sqrt{k})$. In the 
smaller ball $B(x_i,0.5\varepsilon/\sqrt{k})$ we have $d(x,E_k)\ge 
0.5\varepsilon/\sqrt{k}$. Thus

\begin{align*}
 \int_{X}d(x,E_k)\,d\sigma 
 	&\ge \sum_i \int_{B(x_i,\varepsilon/\sqrt{k})} d(x,E_k)\,d\sigma 
 	\gtrsim \sum_i \frac 1{\sqrt{k}}\sigma\big( B(x_i, \varepsilon/\sqrt{k}) \big)\\
 &\gtrsim \#E_k \frac 1{\sqrt{k}} k^{-n}\simeq \frac 1{\sqrt{k}}.
 \end{align*}

\end{remark}

\begin{proof}[Proof of Theorem~\ref{main}]
To prove this we provide a (complex) transport plan between the probability 
measure $b_k(x)=\frac 1{N_k} K_k(x,x)$ -- $b_k$ stands for \emph{Bergman 
measure} -- and the empirical measure $\mu_k$. We are going to prove that 
\[
\mathbb{E}W(\mu_k,b_k)=O\Big(\frac 1{\sqrt{k}}\Big)\ .
\]
Standard estimates for the Bergman kernel provide:
\[
W(b_k,\nu)=O\Big(\frac 1{\sqrt{k}}\Big).
\]
Actually one can prove that the total variation 
(which dominates the Kantorovich-Wasserstein distance) satisfies:
\begin{equation}\label{tvariation}
 \left\|\frac{K_k(x,x)}{N_k} -\nu \right\|\le
\frac{C}{\sqrt{k}}.
\end{equation}
This follows for instance by the expansion in powers of $1/k$ of the Bergman 
kernel.  In this context this is due to Tian, Catlin and Zelditch, 
\cites{Tian90,Catlin97,Zelditch98}.

In the particular case of the spherical ensemble, the kernel is explicit and 
invariant under rotations, and the estimate is even better: the Bergman measure 
is the equilibrium measure, i.e. $b_k=\nu$.

Consider the transport plan
\[
 p(x,y)=\frac 1{N_k}\sum_{j=1}^{N_k} \delta_{x_j}(y)\langle 
K_k(x,x_j),\ell_j(x)\rangle \,d\sigma(x).
\]
It has the correct marginals -- $b_k$ and $\mu_k$ respectively -- and thus
\begin{align*}
 W(b_k,\mu_k)&\le \iint_{X\times X} d(x,y) \,d|p|(x,y)\\ 
 &\leq \frac  1{N_k}\sum_{j=1}^{N_k}  \int_X  d(x,x_j) |\ell_j(x)| 
|K_k(x,x_j)|\,d\sigma(x).
\end{align*}
Now, letting $\beta'$ be the conjugate exponent of $\beta$ (so that 
$1/\beta+1/\beta' =1$), we have
\begin{multline*}
 (\mathbb{E} W)^\beta \le\\
\leq \int_{X^{N_k}}  \frac 1{N_k}\sum_{j=1}^{N_k}\left(
\int_{X}
 d(x,x_j) |\ell_j(x)| |K_k(x,x_j)|d\sigma(x)
\right)^\beta\rho_k(x_1,\ldots,x_{N_k}) \\
\int_{X^{N_k}}  \frac 1{N_k}\sum_{j=1}^{N_k}
\left(
\int_{X}
 d(x,x_j) |K_k(x,x_j)|\right)^{\beta/\beta'}
 \left(
\int_{X}
|\ell_j(x)|^\beta|K_k(x,x_j)|d(x,x_j)\right)\rho_k(x_1,\ldots,x_{N_k}).
\end{multline*}

Assume for the moment that the following off-diagonal decay of the Bergman 
kernel holds:
\begin{equation}\label{offdiagonal}
 \sup_{y\in X}\int_{X}  d(x,y) |K_k(x,y)|\,d\sigma(x)\le
\frac C{\sqrt{k}}.
\end{equation}
Then, by \eqref{lagra}, we obtain:
\begin{align*}
(\mathbb{E} W)^\beta&\le \left(\frac C{\sqrt{k}}\right)^{\beta/\beta'} 
\int_{X}  \frac 
1{N_k}\sum_{j=1}^{N_k}
\int_{X}
|\ell_j(x)|^\beta|K_k(x,x_j)|d(x,x_j)\rho_k(x_1,\ldots,x_j,\ldots,x_{N_k})\\
&=\left(\frac C{\sqrt{k}}\right)^{\beta/\beta'}  \int_{X^{N_k}} 
\frac 1{N_k}\sum_{j=1}^{N_k}
\int_{X}
|K_k(x,x_j)|d(x,x_j)\rho_k(x_1,\ldots,x,\ldots,x_{N_k}).
\end{align*}
Finally, integrating first in $x_j$ and applying again \eqref{offdiagonal}  we 
obtain 
 \[
  (\mathbb{E} W)^\beta\le \Bigl(\frac C{\sqrt{k}}\Bigr)^{\beta/\beta'} 
\Bigl(\frac C{\sqrt{k}}\Bigr)
  = \textrm{O}\Bigl(\frac  1{\sqrt{k}}\Bigr)^{\beta},
 \]
 as desired.

Estimate \eqref{offdiagonal}  follows from the pointwise  estimate for the 
Bergman kernel 
\begin{equation}\label{offidagonal}
 |K_k(x,y)|\le C N_k e^{-C \sqrt{k}\, d(x,y)},
\end{equation}
which holds when the line bundle is positive, see \cite{Berndtsson03}. 

Indeed, consider the function $h(s)=s e^{-C \sqrt{k} s}$ strictly decreasing in 
$\big[\frac{1}{C \sqrt{k}},+\infty\big).$  For any $y\in X$ we bound the 
integral in (\ref{offdiagonal}) as
\begin{align*}
\int_{X} &  d(x,y) |K_k(x,y)|\,d\sigma(x)\lesssim  \int_0^{+\infty} 
\sigma\left( \{x\in X : h(d(x,y))>s\} \right) \,ds  
\\
&
\lesssim N_k \int_{(C \sqrt{k})^{-1}}^{+\infty} |h'(s)| \sigma\left( \{x\in X : 
d(x,y)<s\} \right) \,ds\lesssim \frac{1}{\sqrt{k}},
\end{align*}
where the last estimate follows from $\sigma(B(y,s))\lesssim s^{2n}$ and 
$N_k\sim k^n.$

In the particular case of the spherical ensemble, the kernel is explicit and the 
decay is even faster:

\begin{align*}
|K_k(z,w)|^2&={k^2}\left(1-\frac{|z-w|^2}{(1+|z|^2)(1+|w|^2)}\right)^{k-1}\\
&\le K k^2 \exp\left(-C k \frac{|z-w|^2}{(1+|z|^2)(1+|w|^2)}\right)=\\
&=K k^2 \exp\left( -C k\, d(z,w)^2\right),
\end{align*}
where here $d(z,w)$ coincides with the chordal metric.
\end{proof}

\section{The determinantal setting}

Now we turn our attention to the almost sure convergence of the empirical
measure. Using the fact that Lipschitz functionals of determinantal
process concentrate the measure around the mean we prove the following result.

\begin{coro}
 If $\mu_k$ is the empirical measure associated with the determinantal point 
process given by \eqref{distribution} with $\beta=2$, and $\nu$ denotes the 
equilibrium measure, then 
\begin{itemize}
 \item If  $\dim_{\mathbb C}(M)>1$ then $W(\mu_k,\nu)=O(1/\sqrt{k})$ almost surely.
 \item If $\dim_{\mathbb  C}(M)=1$ then
$W(\mu_k,\nu)=O(\log k/\sqrt{k})$ almost surely.
\end{itemize}
In particular, any realization of the spherical ennsemble 
satisfies  $W(\mu_k,\nu)=O(\log k/\sqrt{k})$ almost surely.
\end{coro}

Let $\nu$ be, as before, the normalized equilibrium measure. Let us define the 
functional $f$ on the set of measures of the form 
$\sigma=\sum_{i=1}^n\delta_{x_i}$  by
\[
 f(\sigma)=n W \Big(\frac \sigma{n} ,\nu\Big). 
\]
As the Kantorovich-Wasserstein distance is controlled by the total variation, 
$f$ is a Lipschitz functional with Lipschitz norm one with respect to the total
variation distance. Here we use the following
result of Pemantle and Peres \cite{PePe14}*{Theorem 3.5}.

\begin{thm*}[Pemantle-Peres]
 Let $Z$ be a determinantal point process of $N$ points. Let $f$ be a 
Lipschitz-1 functional defined in the set of finite counting measures (with 
respect to the total variation distance). Then
\[
\mathbb P (f-\mathbb E f\ge a)\le 3 \exp\left(-\frac{a^2}{16(a+2N)}\right)
\]
\end{thm*}

Take now $a=10 \alpha_k N_k/\sqrt{k}$, where $\alpha_k=\sqrt{\log{k}}$ for 
$n=1$ and $\alpha_k=1$ for $n>1$. Then
\begin{align*}
\mathbb P \Bigl(W(\mu_k,\nu)> \frac{11\alpha_k}{\sqrt{k}}\Bigr)
&\le\mathbb P \Bigl(N_k W(\mu_k,\nu)> N_k \mathbb E W(\mu_k,\mu)+
10\alpha_k\frac{N_k}{\sqrt{k}}\Bigr)\\ 
&\le 3\exp\left(-\frac{100 \alpha_k^2
N_k^2/k}{16(10\alpha_k N_k/\sqrt{k}	+2N_k)}\right)\\
&\lesssim
\exp(-\alpha_k^2 N_k/k)\lesssim\frac1{k^2}.
\end{align*}

Finally, a standard application of the Borel-Cantelli lemma shows that, with
probability one,
\[
 W(\mu_k,\nu)\le \frac{11\sqrt{\alpha_k}}{\sqrt{k}}.
\]

\end{document}